\documentclass[12pt,leqno]{amsart}
\usepackage{amsmath}
\usepackage{amssymb}
\def\overset#1#2{{\mathrel{\mathop {{#2}_{}}\limits^{#1}}}}
\def\underset#1#2{{\mathrel{\mathop {{}_{} {#2}}\limits_{{#1}_{}}}}}
\def\upplim_#1{\underset{#1}{\overline\lim}\;}
\def\lowlim_#1{\underset{#1}{\underline\lim}\;}
\def\leq{\leqslant}
\def\le{\leqslant}
\newtheorem{corollary}[equation]{Corollary}
\newtheorem{definition}[equation]{\indent{\it Definition}\rm }

\newtheorem{lemma}[equation]{Lemma}

\newtheorem{theorem}[equation]{Theorem}

\renewcommand{\dim}{{\mathrm{dim}}}
\newcommand{\zero}{\mathrm{Zero}}
\newcommand{\rank}{\mathrm{rank}}
\newcommand{\C}{{\mathbb{C}}}

\renewcommand{\P}{{\mathbb{P}}}

\begin{document}
\date{} 
\title[Second main theorem and unicity of meromorphic mappings]{Second main theorem and unicity of meromorphic mappings for hypersurfaces in projective varieties} 

\author{\sc Si Duc Quang}
\address{Department of Mathematics\\
 Hanoi National University of Education\\
136-Xuan Thuy, Cau Giay, Hanoi, Vietnam}
\email{quangsd@hnue.edu.vn}

\author{\sc Do Phuong An}
\address{Division of Mathematics\\
Banking Academy\\
12-Chua Boc, Dong Da, Hanoi, Vietnam}
\email{phuongan89@gmail.com}

\thanks{This research is funded by Vietnam National Foundation for Science and Technology Development (NAFOSTED) under grant number 101.04-2015.03.}


\maketitle      
\begin{abstract}   
Let $V$ be a projective subvariety of $\mathbb P^n(\mathbb C)$. A family of hypersurfaces $\{Q_i\}_{i=1}^q$ in $\mathbb P^n(\mathbb C)$ is said to be in $N$-subgeneral position with respect to $V$ if for any $1\le i_1<\cdots <i_{N+1}\le q$,
$ V\cap (\bigcap_{j=1}^{N+1}Q_{i_j})=\varnothing$. In this paper, we will prove a second main theorem for meromorphic mappings of $\mathbb C^m$ into $V$ intersecting hypersurfaces in subgeneral position with truncated counting functions.  As an application of the above theorem,  we give a uniqueness theorem for meromorphic mappings of $\mathbb C^m$ into $V$ sharing a few hypersurfaces without counting multiplicity. In particular, we extend the uniqueness theorem for linear nondegenerate meromorphic mappings of $\C^m$ into $\P^n(\C)$ sharing $2n+3$ hyperplanes in general position to the case where the mappings may be linear degenerate.
\end{abstract}

\section{Introduction and Main results} 

This article is a continuation of our studies in  \cite{AQT}.  To formulate the main result in \cite{AQT}, we recall the following.

Let $N\geq n$ and $q\geq N+1.$ Let $D_1,\cdots, D_q$ be hypersurfaces in $\mathbb P^n(\mathbb C).$
The hypersurfaces $D_1,\cdots, D_q$ are said to be in $N$-subgeneral position in $\mathbb P^n(\mathbb C)$ if 
$D_{j_0}\cap\cdots\cap D_{j_N}=\varnothing$ for every $1\leq j_0<\cdots<j_N\leq q.$ 

Throughout this paper, sometimes we will identify a hypersurface in $\P^n(\C)$ with one of its defining homogeneous polynomials if there is no confusion. In \cite{AQT}, the authors proved the following result.

\begin{theorem}\label{1.1} 
Let $f$ be an algebraically nondegenerate meromorphic mapping of $\mathbb C^m$ into $\mathbb P^n(\mathbb C)$. Let $\{Q_i\}_{i=1}^q$ be hypersurfaces of $\mathbb P^n(\mathbb C)$ in $N$-subgeneral position with $\deg Q_i=d_i$ $(1\le i\le q)$. Let $d=lcm (d_1,\ldots,d_q)$ and $M=\binom{n+d}{n}-1$. Assume that  $q>\frac{(M+1)(2N-n+1)}{n+1}.$ Then, we have
$$\Big\Vert \left (q-\dfrac{(M+1)(2N-n+1)}{n+1}\right )T_f(r)\le \sum_{i=1}^{q}\dfrac{1}{d_i}N^{[M]}_{Q_i(f)}(r)+o(T_f(r)).$$
\end{theorem}

The first aim of this article is to generalize the above Second Main Theorem to meromorphic mappings  into  projective varieties sharing hypersurfaces in subgeneral position. 

We now give the following.

\begin{definition}
Let $V$ be a complex projective subvariety of $\mathbb P^n(\mathbb C)$ of dimension $k\ (k\le n)$. Let $Q_1,...,Q_q\ (q\ge k+1)$ be $q$ hypersurfaces in $\mathbb P^n(\mathbb C)$. The family of hypersurfaces $\{Q_i\}_{i=1}^q$ is said to be in $N$-subgeneral position with respect to $V$ if for any $1\le i_1<\cdots <i_{N+1}\le q$,
$$ V\cap (\bigcap_{j=1}^{N+1}Q_{i_j})=\varnothing .$$
\end{definition}

If  $\{D_i\}_{i=1}^q$ is in $n$-subgeneral position then we say that it is in \textit{general position} with respect to $V.$

Now, let $V$ be a complex projective subvariety of $\mathbb P^n(\mathbb C)$ of dimension $k\ (k\le n)$. Let $d$ be a positive integer. We denote by $I(V)$ the ideal of homogeneous polynomials in $\mathbb C [x_0,...,x_n]$ defining $V$ and by $H_d$ the $\mathbb C$-vector space of all homogeneous polynomials in $\mathbb C [x_0,...,x_n]$ of degree $d.$  Define 
$$I_d(V):=\dfrac{H_d}{I(V)\cap H_d}\text{ and }H_V(d):=\dim I_d(V).$$
Then $H_V(d)$ is called the Hilbert function of $V$. Each element of $I_d(V)$ which is an equivalent class of an element $Q\in H_d,$ will be denoted by $[Q]$, 

\begin{definition}
Let $f:\mathbb C^m\longrightarrow V$ be a meromorphic mapping. We say that $f$ is degenerate over $I_d(V)$ if there is $[Q]\in I_d(V)\setminus \{0\}$ such that $Q(f)\equiv 0.$ Otherwise, we say that $f$ is nondegenerate over $I_d(V)$. It is clear that if $f$ is algebraically nondegenerate, then $f$ is nondegenerate over $I_d(V)$ for every $d\ge 1.$
\end{definition}

Our main theorem is stated as follows. 
\begin{theorem}\label{1.1} 
Let $V$ be a complex projective subvariety of $\mathbb P^n(\mathbb C)$ of dimension $k\ (k\le n)$.
Let $\{Q_i\}_{i=1}^q$ be hypersurfaces of $\mathbb P^n(\mathbb C)$ in $N$-subgeneral position with respect to $V$ with $\deg Q_i=d_i\ (1\le i\le q)$. Let $d$ be the least common multiple of $d_i'$s, i.e., $d=lcm (d_1,...,d_q)$. Let $f$ be a meromorphic mapping of $\mathbb C^m$ into $V$ such that $f$ is nondegenerate over $I_d(V)$. Assume that  $q>\dfrac{(2N-k+1)H_{V}(d)}{k+1}.$ Then, we have
$$ \biggl |\biggl |\ \left (q-\dfrac{(2N-k+1)H_{V}(d)}{k+1}\right )T_f(r)\le \sum_{i=1}^{q}\dfrac{1}{d_i}N^{[H_{V}(d)-1]}_{Q_i(f)}(r)+o(T_f(r)).$$
\end{theorem}
We note that, the second main theorem for algebraically nondegenerate meromorphic mappings into projective subvarieties was firstly given by Min Ru \cite{R09} in 2004. In his result the family of hypersurfaces is assumed in general position and there is no trucation level for the counting functions, but the total defect is $n+1$, which is the sharp number.

\noindent
\textbf{Remark:}

\noindent
(i)\ In the case where $V$ is a linear space of dimension $k$ and each $H_i$ is a hyperplane, i.e., $d_i=1\ (1\le i\le q)$, then $H_V(d)=k+1$ and Theorem \ref{1.1} gives us the classical Second Main Theorem of Cartan-Nochka (see \cite{Noc83} and \cite{No05}). 

\noindent
(ii)\ It is easy to see that $H_V(d)-1\leq \binom{n+d}{n}-1.$ Furthermore,  the truncated level $(H_V(d)-1)$ of the counting function in Theorem \ref{1.1} is much smaller than the previous results of all other authors (cf. \cite{AP}, \cite{DR}). 

\noindent
(iii)\ By a direct computation from Theorem \ref{1.1}, it is easy to see that the total defect is $\dfrac{(2N-k+1)H_V(d)}{k+1}.$ Unfortunately,  this defect is $\ge n+1.$ 

\noindent
(iv)\ Also the above notion of $N$-subgeneral position is a natural generalization from the case of hyperplanes. Therefore, in order to prove Theorem \ref{1.1},  we give a generalization of Nochka weights for hypersurfaces in complex projective varieties.

\noindent
(v)\ From Cartan-Nochka's theorem, we may obtain a second main theorem by using Veronese embedding which embeds $\P^n(\C)$ into $\P^{\binom{n+d}{n}-1}(\C)$. But in that case we need the condition that the family of hyperplanes corresponding to the initial family of hypersurfaces is still in subgeneral position in $\P^{\binom{n+d}{n}-1}(\C)$, which is not satisfied if $N< {\binom{n+d}{n}}$. 

As an application of Theorem \ref{1.1},  the second aim of this article is to give a uniqueness theorem for meromorphic mappings of $\mathbb C^m$ into $V$ sharing a few hypersurfaces without counting multiplicity.

\begin{theorem}\label{1.2}
Let $V$ be a complex projective subvariety of $\mathbb P^n(\mathbb C)$ of dimension $k\ (k\le n)$.  Let $\{Q_i\}_{i=1}^q$ be hypersurfaces in $\mathbb P^n(\mathbb C)$  in $N$-subgeneral position with respect to $V$ and $\deg Q_i=d_i\ (1\le i\le q)$. Let $d$ be the least common multiple of $d_i'$s, i.e., $d=lcm (d_1,...,d_q)$. Let $f$ and $g$ be meromorphic mappings of $\mathbb C^m$ into $V$ which are nondegenerate  over $I_d(V)$.  Assume that

(i) $\dim (\zero Q_i(f)\cap \zero Q_i(f)) \le m-2$ for every $1\le i<j\le q,$

(ii) $f=g$ on $\bigcup_{i=1}^q(\zero Q_i(f)\cup \zero Q_i(g)).$

\noindent
Then the following assertions hold:

a) If $q>\dfrac{2(H_V(d)-1)}{d}+\dfrac{(2N-k+1)H_V(d)}{k+1},$ then $f=g.$

b) If $q>\dfrac{2(2N-k+1)H_{V}(d)}{k+1}$, then there exist $N+1$ hypersurfaces $Q_{i_0},...,Q_{i_N}$, $1\le i_0<\cdots <i_N\le q$, such that
$$ \dfrac{Q_{i_0}(f)}{Q_{i_0}(g)}=\cdots = \dfrac{Q_{i_N}(f)}{Q_{i_N}(g)}.$$
\end{theorem}
\noindent
{\bf N.B.}

\noindent
(i)\ Since the truncated level of the counting function in Theorem \ref{1.1} is better, the number of hypersurfaces in Theorem \ref{1.2} is much smaller than the previous results on unicity of meromorphic mappings sharing hypersurfaces (cf. \cite{DR}, \cite{DT02}).

\noindent
(ii)\ In the case where $d=1$, Theorem \ref{1.2}b) immedietely gives us the following uniqueness theorem for meromorphic mappings into $\P^n(\C)$, which may be linearly degenerate, sharing few hyperplanes in general position.

\begin{corollary}\label{1.5}
 Let $\{H_i\}_{i=1}^{q}$ be hyperplanes in $\mathbb P^n(\mathbb C)$  in general position. Let $f$ and $g$ be meromorphic mappings of $\mathbb C^m$ into $\mathbb P^n(\mathbb C)$.  Assume that

(i) $\dim (\zero H_i(f)\cap \zero H_i(f)) \le m-2$ for every $1\le i<j\le q,$

(ii) $f=g$ on $\bigcup_{i=1}^{q}(\zero H_i(f)\cup \zero H_i(g)).$

\noindent
Let $k$ be the dimension of the smallest linear subspace containing $f(\C^m)$. If $q>2(2n-k+1)$ then $f=g.$
\end{corollary}

We may see that if $f$ is linear nondegenerate, i.e., $k=n$, then the condition of the above corollary is satisfied with $q=2n+3$. Therefore, Corollary \ref{1.5} is a natural extension of the uniqueness for linear nondegenerate meromorphic mappings sharing $2n+3$ hyperplanes in $\P^n(\C)$ in general position given by Yan - Chen \cite{CY}. 

\begin{proof}
Let $f=(f_0:\cdots :f_n)$ and $g=(g_0:\cdots :g_n)$ be two reduced representations of $f$ and $g$ respectively. Let $V(f)$ and $V(g)$ be the smallest linear subspaces of $\P^n(\C)$ containing $f(\C^m)$ and $g(\C^m)$ respectively. It is easy to see that $V(f)$ (resp. $V(g)$) is the intersection of all hyperplanes which contain $f(\C^m)$ (resp. $g(\C^m)$). We may consider $f$ (resp. $g$) as a meromorphic mapping into $V(f)$ (resp. $V(g)$) which is nondegenerate over $I_1(V(f))$ (resp. $I_1(V(g))$). Of course, $H_1,...,H_q$ are in $n$-subgeneral position with respect to both $V(f)$ and $V(g)$.

Now let $H$ be a hyperplane in $\P^n(\C)$ such that $f(\C^m)\subset H$. We denoted again by $H$ the homogeneous linear form defining the hyperplane $H$. Suppose  that $g(\C^m)\not\subset H$, i.e., $H(g)\not\equiv 0$. Then we have $H(g)=H(f)=0$ on $\bigcup_{i=1}^q\zero H_i(g),$ and hence
\begin{align*}
T_g(r)&\ge N_{H(g)}(r)\ge\sum_{i=1}^{q}N^{[1]}_{H_i(g)}(r)+o(T_g(r))\\
&\ge \dfrac{1}{H_{V(g)}(1)-1}\sum_{i=1}^{q}N^{[H_{V(g)}(1)-1]}_{H_i(g)}(r)+o(T_g(r))\\
&\ge \dfrac{1}{H_{V(g)}(1)-1}\left (q-2n+(H_{V(g)}-1)-1\right )T_g(r)+o(T_g(r))\\
&\ge\dfrac{H_{V(g)}+1}{H_{V(g)}-1}T_g(r)+o(T_g(r)),
\end{align*}
(here, note that $H_{V(g)}(1)-1=\dim V(g)$ and $q\ge 2n+3$). This is a contradiction. Therefore, $g(\C^m)\subset H$. This implies that $g(\C^m)\subset V(f)$, and hence $V(g)\subset V(f)$. Similarly, we have $V(f)\subset V(g)$. Then $V(f)=V(g)=V$. 
 
We see that $q>\dfrac{2(2n-k+1)H_V(1)}{k+1}$, since $H_V(1)=k+1$. Therefore, from Theorem \ref{1.2} b), there exist $n+1$ hyperplanes $H_{i_0},...,H_{i_n}$, $1\le i_0<\cdots <i_n\le q$ such that
$$ \dfrac{H_{i_0}(f)}{H_{i_0}(g)}=\cdots = \dfrac{H_{i_n}(f)}{H_{i_n}(g)}.$$
This implies that $f=g$.
\end{proof}

{\bf Acknowledgements.} This work was completed while the first author was staying at the Vietnam Institute for Advanced Study in Mathematics 
(VIASM). He would like to thank the institute for the support. This research is funded by Vietnam National Foundation for Science and Technology Development (NAFOSTED) under grant number 101.04-2015.03.

\section{Basic notions and auxiliary results from Nevanlinna theory}

\noindent
{\bf 2.1.}\ We set $||z|| = \big(|z_1|^2 + \dots + |z_m|^2\big)^{1/2}$ for
$z = (z_1,\dots,z_m) \in \mathbb C^m$ and define

\begin{align*}
B(r) := \{ z \in \mathbb C^m : ||z|| < r\},\quad
S(r) := \{ z \in \mathbb C^m : ||z|| = r\}\ (0<r<\infty).
\end{align*}

Define 
$$v_{m-1}(z) := \big(dd^c ||z||^2\big)^{m-1}\quad \quad \text{and}$$
$$\sigma_m(z):= d^c \text{log}||z||^2 \land \big(dd^c \text{log}||z||^2\big)^{m-1}
 \text{on} \quad \mathbb C^m \setminus \{0\}.$$

 For a divisor $\nu$ on $\mathbb C^m$ and for a positive integer $M$ or $M= \infty$, define the counting function of $\nu$ by
$$\nu^{[M]}(z)=\min\ \{M,\nu(z)\},$$
\begin{align*}
n(t) =
\begin{cases}
\int\limits_{|\nu|\,\cap B(t)}
\nu(z) v_{m-1} & \text  { if } m \geq 2,\\
\sum\limits_{|z|\leq t} \nu (z) & \text { if }  m=1. 
\end{cases}
\end{align*}

Similarly, we define \quad $n^{[M]}(t).$

Define
$$ N(r,\nu)=\int\limits_1^r \dfrac {n(t)}{t^{2m-1}}dt \quad (1<r<\infty).$$

Similarly, define  \ $N(r,\nu^{[M]})$
and denote it by \ $N^{[M]}(r,\nu)$.

Let $\varphi : \mathbb C^m \longrightarrow \mathbb C $ be a meromorphic function. Denote by $\nu_\varphi$ the zero divisor of $\varphi$. Define
$$N_{\varphi}(r)=N(r,\nu_{\varphi}), \ N_{\varphi}^{[M]}(r)=N^{[M]}(r,\nu_{\varphi}).$$

For brevity, we will omit the character $^{[M]}$ if $M=\infty$.

\noindent
{\bf 2.2.}\ Let $f : \mathbb C^m \longrightarrow \mathbb P^n(\mathbb C)$ be a meromorphic mapping.
For arbitrarily fixed homogeneous coordinates
$(w_0 : \dots : w_n)$ on $\mathbb P^n(\mathbb C)$, we take a reduced representation
$f = (f_0 : \dots : f_n)$, which means that each $f_i$ is a  
holomorphic function on $\mathbb C^m$ and 
$f(z) = \big(f_0(z) : \dots : f_n(z)\big)$ outside the analytic subset
$\{ f_0 = \dots = f_n= 0\}$ of codimension $\geq 2$.
Set $\Vert f \Vert = \big(|f_0|^2 + \dots + |f_n|^2\big)^{1/2}$.

The characteristic function of $f$ is defined by 
\begin{align*}
T_f(r)= \int\limits_{S(r)} \log\Vert f \Vert \sigma_m -
\int\limits_{S(1)}\log\Vert f\Vert \sigma_m.
\end{align*}

\noindent
\textbf{2.3.}\ Let $\varphi$ be a nonzero meromorphic function on $\mathbb C^m$, which is occasionally regarded as a meromorphic map into $\mathbb P^1(\mathbb C)$. The proximity function of $\varphi$ is defined by
$$m(r,\varphi)=\int_{S(r)}\log \max\ (|\varphi|,1)\sigma_m.$$
The Nevanlinna's characteristic function of $\varphi$ is define as follows
$$ T(r,\varphi)=N_{\frac{1}{\varphi}}(r)+m(r,\varphi). $$
Then 
$$T_\varphi (r)=T(r,\varphi)+O(1).$$
The function $\varphi$ is said to be small (with respect to $f$) if $||\ T_\varphi (r)=o(T_f(r))$.
Here, by the notation ``$|| \ P$''  we mean the assertion $P$ holds for all $r \in [0,\infty)$ excluding a Borel subset $E$ of the interval $[0,\infty)$ with $\int_E dr<\infty$.

\noindent
{\bf 2.4. Lemma on logarithmic derivative} (see \cite[Lemma 3.11]{Shi}). {\it Let $f$ be a nonzero meromorphic function on $\mathbb C^m.$ Then 
$$\biggl|\biggl|\quad m\biggl(r,\dfrac{\mathcal{D}^\alpha (f)}{f}\biggl)=O(\log^+T(r,f))\ (\alpha\in \mathbb Z^m_+).$$}

Repeating the argument in \cite[Proposition 4.5]{Fu}, we have the following.

\noindent
{\bf 2.5. Proposition.}\  {\it Let $\Phi_0,...,\Phi_k$ be meromorphic functions on $\mathbb C^m$ such that $\{\Phi_0,...,\Phi_k\}$ 
are  linearly independent over $\mathbb C.$
Then  there exists an admissible set  
$$\{\alpha_i=(\alpha_{i1},...,\alpha_{im})\}_{i=0}^k \subset \mathbb Z^m_+$$
with $|\alpha_i|=\sum_{j=1}^{m}|\alpha_{ij}|\le k \ (0\le i \le k)$ such that the following are satisfied:

(i)\  $\{{\mathcal D}^{\alpha_i}\Phi_0,...,{\mathcal D}^{\alpha_i}\Phi_k\}_{i=0}^{k}$ is linearly independent over $\mathcal M,$
\ i.e.,  
$$\det{({\mathcal D}^{\alpha_i}\Phi_j)}\not\equiv 0.$$ 

(ii) $\det \bigl({\mathcal D}^{\alpha_i}(h\Phi_j)\bigl)=h^{k+1}\cdot \det \bigl({\mathcal D}^{\alpha_i}\Phi_j\bigl)$ for
any nonzero meromorphic function $h$ on $\mathbb C^m.$}
 
\section{Generalization of Nochka weights}

Let $V$ be a complex projective subvariety of $\mathbb P^n(\mathbb C)$ of dimension $k\ (k\le n)$. Let $\{Q_i\}_{i=1}^q$ be $q$ hypersurfaces in $\mathbb P^n(\mathbb C)$ of the common degree $d$, which are regarded as homogeneous polynomials in variables $(x_0,...,x_n)$. We regard $I_d(V)=\dfrac{H_d}{I(V)\cap H_d}$ as a complex vector space. It is easy to see that
\begin{align*}
 \rank \{Q_i\}_{i\in R}\ge \dim V-\dim (\bigcap_{i\in R}Q_i\cap V).
\end{align*}

Set $\dim (\varnothing)=-1.$ Then, if  $\{Q_i\}_{i=1}^q$ is in $N$-subgeneral position, we have
$$ \rank \{Q_i\}_{i\in R}\ge \dim V-\dim (\bigcap_{i\in R}Q_i\cap V)=k+1$$
for  any subset $R\subset\{1,...,q\}$ with $\sharp R=N+1$.

Taking an $\mathbb C$-basis of $I_d(V)$, we may consider $I_d(V)$ as a $\mathbb C$-vector space $\mathbb C^M$ with $M=H_V(d)$.

Let $\{H_i\}_{i=1}^q$ be $q$ hyperplanes in $\mathbb C^M$ passing through the coordinates origin. Assume that each $H_i$ is defined by the linear equation
$$ a_{ij}z_1+\cdots +a_{iM}z_M=0,$$
where $a_{ij}\in\mathbb C\ (j=1,...,M ),$ not all zeros. We define the vector associated with $H_i$ by
$$ v_i=(a_{i1},...,a_{iM})\in\mathbb C^M.$$
For each subset $R\subset\{1,...,q\}$, the \textit{rank} of $\{H_i\}_{i\in R}$ is defined by
$$ \rank \{H_i\}_{i\in R}=\rank \{v_i\}_{i\in R}. $$
Recall that the family $\{H_i\}_{i=1}^q$ is said to be in \textit{$N$-subgeneral position} if for any subset $R\subset\{1,...,q\}$ with $\sharp R=N+1$, $\bigcap_{i\in R}H_i=\{0\}$, i.e., $\rank \{H_i\}_{i\in R}=M.$

By Lemmas 3.3 and 3.4 in \cite{No05}, we have the following.

\begin{lemma}\label{3.1}\ Let $\{H_i\}_{i=1}^q$ be $q$ hyperplanes in $\mathbb C^{k+1}$ in $N$-subgeneral position, and assume that $q> 2N-k+1$. 
Then there are positive rational constants $\omega_i\ (1\le i\le q)$ satisfying the following:

i) $0<\omega_j \le 1,\  \forall i\in\{1,...,q\}$,

ii) Setting $\tilde \omega =\max_{j\in Q}\omega_j$, one gets
$$\sum_{j=1}^{q}\omega_j=\tilde \omega (q-2N+k-1)+k+1.$$

iii) $\dfrac{k+1}{2N-k+1}\le \tilde\omega\le\dfrac{k}{N}.$

iv) For $R\subset Q$ with $0<\sharp R\le N+1$, then $\sum_{i\in R}\omega_i\le\rank\{H_i\}_{i\in R}$.

v) Let $E_i\ge 1\ (1\le i \le q)$ be arbitrarily given numbers. For $R\subset Q$ with $0<\sharp R\le N+1$,  there is a subset $R^o\subset R$ such that $\sharp R^o=\rank \{H_i\}_{i\in R^o}=\rank\{H_i\}_{i\in R}$ and
$$\prod_{i\in R}E_i^{\omega_i}\le\prod_{i\in R^o}E_i.$$
\end{lemma}

The above $\omega_j$ are called \emph{Nochka  weights} and  $\tilde\omega$ is called \emph{Nochka  constant.}

\begin{lemma}[{cf. \cite[Lemma 3.2]{AQT}}]\label{3.2}
Let $H_1,...H_q$ be $q$ hyperplanes in $\mathbb C^M\ (M\ge 2),$ passing through the coordinates origin. Let $k$ be a positive integer such 
that $k\le M$. Then there exists a linear subspace $L\subset \mathbb C^M$ of dimension $k$ such that $L\not\subset H_i\ (1\le i\le q)$ and
$$ \rank \{H_{i_1}\cap L,\dots , H_{i_l}\cap L\}=  \rank \{H_{i_1},\dots ,  H_{i_l}\}$$
for every $1\le l\le k, 1\le i_1<\cdots <i_l\le q.$
\end{lemma}

\begin{lemma}\label{3.3}
Let $V$ be a complex projective subvariety of $\mathbb P^n(\mathbb C)$ of dimension $k\ (k\le n)$. Let $Q_1,...,Q_q$ be $q\ (q>2N-k+1)$ hypersurfaces in $\mathbb P^n(\mathbb C)$ in $N$-subgeneral position with respect to $V$ of the common degree $d.$ Then there are positive rational constants $\omega_i\ (1\le i\le q)$ satisfying the following:

i) $0<\omega_i \le 1,\  \forall i\in\{1,...,q\}$,

ii) Setting $\tilde \omega =\max_{j\in Q}\omega_j$, one gets
$$\sum_{j=1}^{q}\omega_j=\tilde \omega (q-2N+k-1)+k+1.$$

iii) $\dfrac{k+1}{2N-k+1}\le \tilde\omega\le\dfrac{k}{N}.$

iv) For $R\subset \{1,...,q\}$ with $\sharp R = N+1$, then $\sum_{i\in R}\omega_i\le k+1$.

v) Let $E_i\ge 1\ (1\le i \le q)$ be arbitrarily given numbers. For $R\subset \{1,...,q\}$ with $\sharp R = N+1$,  there is a subset $R^o\subset R$ such that $\sharp R^o=\rank \{Q_i\}_{i\in R^o}=k+1$ and 
$$\prod_{i\in R}E_i^{\omega_i}\le\prod_{i\in R^o}E_i.$$
\end{lemma}
\begin{proof}
We assume that each $Q_i$ is given by
$$\sum_{I\in\mathcal I_d}a_{iI}x^I=0, $$
where $\mathcal I_d=\{(i_0,...,i_n)\in \mathbb N_0^{n+1}\ :\ i_0+\cdots + i_n=d\}$, $I=(i_0,...,i_n)\in\mathcal I_d$, $x^I=x_0^{i_0}\cdots x_n^{i_n}$ and $a_{iI}\in \mathbb C\ (1\le i\le q, I\in\mathcal I_d)$. Setting $Q^*_i(x)=\sum_{I\in\mathcal I_d}a_{iI}x^I.$ Then $Q^*_i\in H_d.$

Taking a $\mathbb C$-basis of $I_d(V)$, we may identify $I_d(V)$ with the $\mathbb C$-vector space $\mathbb C^M,$ where $M=H_V(d)$. For each $Q_i$, denote by $v_i$ the vector in $\mathbb C^M$ which corresponds to $[Q_i^*]$ by this identification. Denote by $H_i$ the hyperplane in $\mathbb C^{M}$ associated with  the vector $v_i$.

Then for each arbitrary subset $R\subset\{1,...,q\}$ with $\sharp R=N+1$, we have 
$$ \dim (\bigcap_{i\in R}Q_i\cap V)\ge  \dim V-\rank \{[Q_i]\}_{i\in R}=k-\rank \{H_i\}_{i\in R}.$$
Hence
$$\rank \{H_i\}_{i\in R}\ge k- \dim (\bigcap_{i\in R}Q_i\cap V)\ge k-(-1)=k+1. $$

By Lemma \ref{3.2}, there exists a linear subspace $L\subset\mathbb C^{M}$ of dimension $k+1$ such that $L\not\subset H_i\ (1\le i\le q)$ and 
$$ \rank \{H_{i_1}\cap L,\dots , H_{i_l}\cap L\}=  \rank \{H_{i_1},\dots ,  H_{i_l}\}$$
for every $1\le l\le k+1, 1\le i_1<\cdots <i_l\le q.$  Since $\rank \{H_i\}_{i\in R}\ge k+1,$ it implies that for any subset $R\in\{1,...,q\}$ with $\sharp R=N+1,$ there exists a subset $R'\subset R$ with $\sharp R'=k+1$ and $\rank \{H_i\}_{i\in R'}=k+1.$ Hence, we get
$$ \rank\{H_i\cap L\}_{i\in R}\ge \rank\{H_i\cap L\}_{i\in R'}= \rank\{H_i\}_{i\in R'}=k+1.$$
This yields that $\rank\{H_i\cap L\}_{i\in R}=k+1,$ since $\dim L=k+1$. Therefore, $\{H_i\cap L\}_{i=1}^q$ is a family of $q$ hyperplanes in $L$ in $N$-subgeneral position.

By Lemma \ref{3.1}, there exist Nochka weights $\{\omega_i\}_{i=1}^q$ for the family $\{H_i\cap L\}_{i=1}^q$ in $L.$ It is clear that assertions (i)-(iv) are automatically satisfied. Now for $R\subset\{1,...,q\}$ with $\sharp R=N+1$, by Lemma \ref{3.1}(v) we have 
$$ \sum_{i\in R}\omega_i\le\rank \{H_i\cap L\}_{i\in R}=k+1 $$
and there is a subset $R^o\subset R$ such that: 
\begin{align*}
&\sharp R^o=\rank \{H_i\cap L\}_{i\in R^0}=\rank \{H_i\cap L\}_{i\in R}=k+1,\\
& \prod_{i\in R}E_i^{\omega_i}\le\prod_{i\in R^o}E_i, \ \ \forall E_i\ge 1 \ (1\le i\le q),\\
&\rank \{Q_i\}_{i\in R^0}=\rank \{H_i\cap L\}_{i\in R^0}=k+1.
\end{align*}
Hence the assertion (v) is also satisfied.
The lemma is proved.
\end{proof}

\section{Second main theorems for hypersurfaces}

Let $\{Q_i\}_{i\in R}$ be a set of hypersurfaces in $\mathbb P^n(\mathbb C)$ of the common degree $d$. Assume that each $Q_i$ is defined by
$$ \sum_{I\in\mathcal I_d}a_{iI}x^I=0, $$
where $\mathcal I_d=\{(i_0,...,i_n)\in \mathbb N_0^{n+1}\ :\ i_0+\cdots + i_n=d\}$, $I=(i_0,...,i_n)\in\mathcal I_d,$ $x^I=x_0^{i_0}\cdots x_n^{i_n}$ and $(x_0:\cdots: x_n)$ is homogeneous coordinates of $\mathbb P^n(\mathbb C)$.

Let $f:\mathbb C^m\longrightarrow V\subset\mathbb P^n(\mathbb C)$ be an algebraically nondegenerate meromorphic mapping into $V$ with a reduced representation $f=(f_0:\cdots :f_n)$. We define
$$ Q_i(f)=\sum_{I\in\mathcal I_d}a_{iI}f^I ,$$
where $f^I=f_0^{i_0}\cdots f_n^{i_n}$ for $I=(i_0,...,i_n)$. Then we see that $f^*Q_i=\nu_{Q_i(f)}$ as divisors.

\begin{lemma}\label{4.1}
Let $\{Q_i\}_{i\in R}$ be a set of hypersurfaces in $\mathbb P^n(\mathbb C)$ of the common degree $d$ and let $f$ be a meromorphic mapping of $\mathbb C^m$ into $\mathbb P^n(\mathbb C)$. Assume that $\bigcap_{i\in R}Q_i\cap V=\varnothing$. Then there exist positive constants $\alpha$ and $\beta$ such that
$$\alpha ||f||^d \le  \max_{i\in R}|Q_i(f)|\le \beta ||f||^d.$$
\end{lemma} 
\begin{proof}  
Let $(x_0:\cdots: x_n)$ be homogeneous coordinates of $\mathbb P^n(\mathbb C)$. Assume that each $Q_i$ is defined by $\sum_{I\in\mathcal I_d}a_{iI}x^I=0.$ 

Set $Q_i(x)=\sum_{I\in\mathcal I_d}a_{iI}x^I$ and consider the following function
$$ h(x)=\dfrac{\max_{i\in R}|Q_i(x)|}{||x||^d}, $$
where $||x||=(\sum_{i=0}^n|x_i|^2)^{\frac{1}{2}}$.

Since the function $h$ is positive continuous on $V,$ by the compactness of  $V$, there exist positive constants $\alpha$ and $\beta$ such that $\alpha =\min_{x\in \mathbb P^n(\mathbb C)}h(x)$ and $\beta =\max_{x\in \mathbb P^n(\mathbb C)}h(x)$. Thus
$$\alpha ||f||^d \le  \max_{i\in R}|Q_i(f)|\le \beta ||f||^d.$$
The lemma is proved. 
\end{proof}

The following lemma is due to Lemma 4.2 in \cite{AQT} with a slightly modification.
\begin{lemma}[{cf. \cite[Lemma 4.2]{AQT}}]\label{4.2}
Let $\{Q_i\}_{i=1}^q$ be a set of $q$ hypersurfaces in $\mathbb P^n(\mathbb C)$ of the common degree $d$. Then there exist $(H_V(d)-k-1)$ hypersurfaces $\{T_i\}_{i=1}^{H_V(d)-k-1}$ in $\mathbb P^n(\mathbb C)$ such that for any subset $R\in\{1,...,q\}$ with $\sharp R=\rank \{Q_i\}_{i\in R}=k+1,$ we get $\rank \{\{Q_i\}_{i\in R}\cup\{T_i\}_{i=1}^{M-k}\}=H_V(d).$
\end{lemma}
\begin{proof} For each$R\subset\{1,...,q\}$ with $\sharp R=\rank \{Q_i\}_{i\in R}=k+1$, denote by $V_R$ the set of all vectors $v=(v_1,...,v_{H_V(d)-k-1})\in (I_d(V))^{H_V(d)-k-1}$ such that $\{ \{[Q_i]\}_{i\in R}, v_1,...,v_{H_V(d)-k-1}\}$ is linearly dependent over $\mathbb C$. Then $V_R$ is an algebraic subset of $(I_d(V))^{H_V(d)-k-1}$. Since $\dim I_d(V)=H_V(d)$ and  $\rank \{Q_i\}_{i\in R}=k+1$, there exists an element 
$$v=(v_1,...,v_{H_V(d)-k-1})\in (I_d(V))^{H_V(d)-k-1}$$
such that the family of vectors $\{ \{[Q_i]\}_{i\in R}, v_1,...,v_{H_V(d)-k-1}\}$ is linearly independent over $\mathbb C$, i.e., $v\not\in V_R$. Therefore $V_R$ is a proper algebraic subset of $(I_d(V))^{H_V(d)-k-1}$ for each $R.$ This implies that
$$ (I_d(V))^{H_V(d)-k-1}\setminus\bigcup_{R} V_R\ne\varnothing .$$
Hence, there is $(T^+_1,...,T^+_{H_V(d)-k-1})\in (I_d(V))^{H_V(d)-k-1}\setminus\bigcup_{R} V_R$.

For each $T^+_i$, take a representation $T_i\in H_d$ of $T^+_i$. Then
\begin{align*}
\rank \{\{Q_i\}_{i\in R}\cup\{T_i\}_{i=1}^{H_V(d)-k-1}\}=\rank\{\{[Q_i]\}_{i\in R}\cup\{[T_i]\}_{i=1}^{H_V(d)-k-1}\}=H_V(d)
\end{align*}
for every subset $R\in\{1,...,q\}$ with $\sharp R=\rank \{Q_i\}_{i\in R}=k+1$.

The lemma is proved.
\end{proof}

\vskip0.2cm
\noindent
{\bf Proof of Theorem \ref{1.1}.}\\
We first prove the theorem in the case where all $Q_i\ (i=1,...,q)$ do have the same degree $d$.
It is easy to see that there is a positive constant $\beta$ such that $\beta ||f||^d\ge |Q_i(f)|$ for every $1\le i\le q.$
Set $ Q:=\{1,\cdots ,q\}$. Let $\{\omega_i\}_{i=1}^q$ be as in Lemma \ref{3.3} for the family $\{Q_i\}_{i=1}^q$.  Let $\{T_i\}_{i=1}^{M-k}$ be $(M-k)$ hypersurfaces in $\mathbb P^n(\mathbb C)$, which satisfy Lemma \ref{4.2}. 

Take a $\mathbb C$-basis $\{[A_i]\}_{i=1}^{H_V(d)}$ of $I_d(V)$, where $A_i\in H_d$. Since $f$ is nondegenerate over $I_d(V)$, it implies that $\{A_i(f); 1\le i\le H_V(d)\}$ is linearly independent over $\mathbb C$. Then there is an admissible set $\{\alpha_1,\cdots ,\alpha_{H_V(d)}\}\subset  \mathbb Z_+^m$ such that 
$$W\equiv\det\bigl (\mathcal D^{\alpha_j}A_i(f)(1\le i\le H_V(d))\bigl )_{1\le j\le H_V(d)}\not\equiv 0$$
and  $|\alpha_j|\le H_V(d)-1$  for all $1\le j\le H_V(d).$

For each $R^o=\{r^0_1,...,r^0_{k+1}\}\subset\{1,...,q\}$ with $\rank \{Q_i\}_{i\in R^o}=\sharp R^o=k+1$, set $$W_{R^o}\equiv\det\bigl (\mathcal D^{\alpha_j}Q_{r^0_v}(f) (1\le v\le k+1),\mathcal D^{\alpha_j}T_l(f) (1\le l\le H_V(d)-k-1)\bigl )_{1\le j\le H_V(d)}.$$
Since $\rank \{Q_{r^0_v} (1\le v\le k+1),T_l (1\le l\le H_V(d)-k-1)\}=H_V(d)$, there exists a nonzero constant  $C_{R^o}$ such that $W_{R^o}=C_{R^o}\cdot W$. 

We denote by $\mathcal R^o$ the family of all subsets $R^o$ of $\{1,...,q\}$ satisfying $$\rank \{Q_i\}_{i\in R^o}=\sharp R^o=k+1.$$

Let $z$ be a fixed point. For each $R\subset Q$ with  $\sharp R=N+1,$  we choose $R^{o}\subset R$ such that $R^o\in\mathcal R^o$ and $R^o$ satisfies Lemma \ref{3.3} v) with respect to numbers $\bigl \{\dfrac{\beta ||f(z)||^d}{|Q_i(f)(z)|}\bigl \}_{i=1}^q$.  On the other hand, there exists $\bar R\subset Q$ with $\sharp \bar R=N+1$ such that $|Q_{i}(f)(z)|\le |Q_j(f)(z)|,\forall i\in \bar R,j\not\in \bar R$. Since $\bigcap_{i\in \bar R}Q_i=\varnothing$, by Lemma \ref{4.1}, there exists a positive constant $\alpha_{\bar R}$ such that
$$ \alpha_{\bar R} ||f||^d(z)\le \max_{i\in \bar R}|Q_i(f)(z)|. $$
Then, we get
\begin{align*}
\dfrac{||f(z)||^{d(\sum_{i=1}^q\omega_i)}|W(z)|}{|Q_1^{\omega_1}(f)(z)\cdots Q_q^{\omega_q}(f)(z)|}
&\le\dfrac{|W(z)|}{\alpha^{q-N-1}_{\bar R}\beta^{N+1}}\prod_{i\in \bar R}\left (\dfrac{\beta||f(z)||^d}{|Q_i(f)(z)|}\right )^{\omega_i}\\
&\le A_{\bar R}\dfrac{|W(z)|\cdot ||f||^{d(k+1)}(z)}{\prod_{i\in \bar R^o}|Q_i(f)|(z)}\\
&\le B_{\bar R}\dfrac{|W_{\bar R^o}(z)|\cdot ||f||^{dH_V(d)}(z)}{\prod_{i\in \bar R^o}|Q_i(f)|(z)\prod_{i=1}^{H_V(d)-k-1}|T_i(f)|(z)},
\end{align*}
where $A_{\bar R}, B_{\bar R}$ are positive constants. 

Put $S_{\bar R}=B_{\bar R}\dfrac{|W_{\bar R^o}|}{\prod_{i\in \bar R^o}|Q_i(f)|\prod_{i=1}^{H_V(d)-k-1}|T_i(f)|}$. By the Lemma on logarithmic derivative, it is easy to see that 
$$||\ \int_{S(r)}\log^{+}S_{\bar R}(z)\sigma_m=o(T_f(r)).$$

Therefore, for each  $z\in \mathbb C^m$, we have
\begin{align*}
\log \left (\dfrac{||f(z)||^{d(\sum_{i=1}^q\omega_i)}|W(z)|}{|Q_1^{\omega_1}(f)(z)\cdots Q_q^{\omega_q}(f)(z)|}\right )\le \log \left (||f||^{dH_V(d)}(z)\right )+\sum_{R\subset Q,\sharp R=N+1}\log^+S_R.
\end{align*}
Since $\sum_{i=1}^q\omega_i=\tilde\omega_i(q-2N+k-1)+k+1$ and by integrating both sides of the above inequality over $S(r),$  we have
\begin{align}\label{4.5}
||\   d(q-2N+k-1-\dfrac{H_V(d)-k-1}{\tilde\omega})T_f(r)\le\sum_{i=1}^{q}\dfrac{\omega_i}{\tilde\omega}N_{Q_i(f)}(r)-\dfrac{1}{\tilde\omega}N_{W}(r)+o(T_f(r)).
\end{align}

\textbf{Claim.} $\sum_{i=1}^q\omega_iN_{Q_i(f)}(r)-N_{W}(r)\le \sum_{i=1}^q\omega_iN^{[H_V(d)-1]}_{Q_i(f)}(r)$.

Indeed, let $z$ be a zero of some $Q_i(f)(z)$ and $z\not\in I(f)=\{f_0=\cdots =f_n=0\}$. Since $\{Q_i\}_{i=1}^q$ is in $N$-subgeneral position, $z$ is not zero of more than $N$ functions $Q_i(f)$. Without loss of generality, we may assume that $z$ is zero of $Q_i(f)$ for each $1\le i\le k\le N)$ and $z$ is not zero of $Q_i(f)$ for each $i>N$. Put $R=\{1,...,N+1\}.$ Choose $R^1\subset R$ such that 
$\sharp R^1=\rank\{Q_i\}_{i\in R^1}=k+1$ and $R^1$ satisfies Lemma \ref{3.3} v) with respect to numbers $\bigl \{e^{\max\{\nu_{Q_i(f)}(z)-H_V(d)+1,0\}} \bigl \}_{i=1}^q.$ Then we have
\begin{align*}
 \sum_{i\in R}\omega_i \max\{\nu_{Q_i(f)}(z)-H_V(d)+1,0\} \le \sum_{i\in R^1}\max\{\nu_{Q_i(f)}(z)-H_V(d)+1,0\}.
\end{align*}
This yields that
\begin{align*}
\nu_{W}(z)= \nu_{W_{R^1}}(z)&\ge \sum_{i\in R^1}\max\{\nu_{Q_i(f)}(z)-H_V(d)+1,0\}\\
&\ge\sum_{i\in R}\omega_i \max\{\nu_{Q_i(f)}(z)-H_V(d)+1,0\}. 
\end{align*}
Hence 
\begin{align*}
\sum_{i=1}^q\omega_i\nu_{Q_i(f)}(z)-\nu_{W}(z)& =\sum_{i\in R}\omega_i\nu_{Q_i(f)}(z)-\nu_{W}(z)\\ 
&= \sum_{i\in R}\omega_i\min\{\nu_{Q_i(f)}(z),H_V(d)-1\}\\
&+\sum_{i\in R}\omega_i\max\{\nu_{Q_i(f)}(z)-H_V(d)+1,0\}-\nu_{W}(z)\\
&\le \sum_{i\in R}\omega_i\min\{\nu_{Q_i(f)}(z),H_V(d)+1\}\\
&=\sum_{i=1}^q\omega_i\min\{\nu_{Q_i(f)}(z),M\}.
\end{align*}
Integrating both sides of this inequality, we get
$$ \sum_{i=1}^q\omega_iN_{Q_i(f)}(r)-N_{W}(r)\le \sum_{i=1}^q\omega_iN^{[H_V(d)-1]}_{Q_i(f)}(r).$$
This proves the claim.

Combining the claim and (\ref{4.5}), we obtain
\begin{align*}
||\  & d(q-2N+k-1-\dfrac{H_V(d)-k-1}{\tilde\omega})T_f(r)\\
&\le\sum_{i=1}^{q}\dfrac{\omega_i}{\tilde\omega}N^{[H_V(d)-1]}_{Q_i(f)}(r)+o(T_f(r))\\
&\le \sum_{i=1}^{q}N^{[H_V(d)-1]}_{Q_i(f)}(r)+o(T_f(r)).
\end{align*}
Since $\tilde\omega \ge \dfrac{k+1}{2N-k+1}$, the above inequality  implies that
$$ \biggl|\biggl|\quad  d\left (q-\dfrac{(2N-k+1)H_V(d)}{k+1}\right )T_f(r)\le \sum_{i=1}^{q}N^{[H_V(d)-1]}_{Q_i(f)}(r)+o(T_f(r)).$$
Hence, the theorem is proved in the case where all $Q_i$ do have the same degree.

We now prove the theorem in the general case where $\deg Q_i=d_i$. Applying the above case for $f$ and the hypersurfaces 
$Q^{\frac{d}{d_i}}_i\ (i=1,...,q)$ of the common degree $d$, we have
\begin{align*}
 \biggl|\biggl|\quad  \left (q-\dfrac{(2N-k+1)H_V(d)}{k+1}\right )T_f(r)&\le \dfrac{1}{d}\sum_{i=1}^{q}N^{[H_V(d)-1]}_{Q^{d/d_i}_i(f)}(r)+o(T_f(r))\\
&\le \sum_{i=1}^{q}\dfrac{1}{d}\frac{d}{d_i}N^{[H_V(d)-1]}_{Q_i(f)}(r)+o(T_f(r))\\
&=\sum_{i=1}^{q}\dfrac{1}{d_i}N^{[H_V(d)-1]}_{Q_i(f)}(r)+o(T_f(r)).
\end{align*}
The theorem is proved.\hfill$\square$

\section{Unicity of meromorphic mappings sharing hypersurfaces}

\begin{lemma}\label{5.1}
Let $f$ and $g$ be nonconstant meromorphic mappings of $\mathbb C^m$ into a complex projective subvariety $V$ of 
$\mathbb P^n(\mathbb C )$, $\dim V=k\ (k\le n)$. Let $Q_i\ (i=1,...,q)$ be moving hypersurfaces in $\mathbb P^n(\mathbb C)$  in 
$N$-subgeneral position with respect to $V$, $\deg Q_i=d_i$, $N\ge n$. Put $d=lcm (d_1,...,d_q)$ and $M=\binom{n+d}{n}-1$. Assume that both $f$ and $g$ are nondegenerate over $I_d(V)$.  Then $ ||\ T_f(r)=O(T_g(r))\ \text{ and }\  ||\ T_g(r)=O(T_f(r))$ if $q>\frac{(2N-k+1)H_V(d)}{k+1}.$
\end{lemma}
\textbf{\textit{Proof.}}\ Using Theorem \ref{1.1} for $f$, we have
\begin{align*}
\biggl|\biggl|\quad & \left (q-\dfrac{(2N-k+1)H_V(d)}{k+1}\right )T_f(r)\\
\le &\sum_{i=1}^{q}\dfrac{1}{d_i}N^{[H_V(d)-1]}_{Q_i(f)}(r)+o(T_f(r))\\
\le &\sum_{i=1}^q\dfrac{H_V(d)-1}{d_i}\ N_{Q_i(f)}^{[1]}(r)+o(T_f(r))\\
= &\sum_{i=1}^q\dfrac{H_V(d)-1}{d_i}\ N_{Q_i(g)}^{[1]}(r)+o(T_f(r))\\
\le & q(H_V(d)-1)\ T_g(r)+o(T_f(r)).
\end{align*}
Hence \quad $|| \quad T_f(r)=O(T_g(r)).$ Similarly, we get \  \ $|| \ \ T_g(r)=O(T_f(r)).$

\vskip0.2cm
\noindent
\textbf{Proof of Theorem \ref{1.2}.}\\
Assume that $f=(f_0:\cdots :f_n)$ and $g=(g_0:\cdots :g_n)$ are reduced representations of $f$ and $g,$ respectively. Replacing $Q_i$ by $Q_i^{\frac{d}{d_i}}$ if necessary, without loss of generality, we may assume that $d_i=d$ for all $1\le i\le q.$

a) By Lemma \ref{5.1}, we have $ ||\ T_f(r)=O(T_g(r))\ \text{ and }\  ||\ T_g(r)=O(T_f(r)).$ Suppose that $f\ne g$. Then there exist two indices 
$s,t$ with $0\le s<t\le n$ such that $ H:=f_sg_t-f_tg_s\not\equiv 0. $
By the assumption (ii) of the theorem, we have $H=0$ on $\bigcup_{i=1}^q(\zero Q_i(f)\cup\zero Q_i(g))$. Therefore, we have
$$ \nu^0_H\ge \sum_{i=1}^q\min\{1,\nu^0_{Q_i(f)}\}$$
outside an analytic subset of codimension at least two. This follows that
\begin{align}\label{5.2}
N_H(r)\ge \sum_{i=1}^qN^{[1]}_{Q_i(f)}(r).
\end{align}

On the other hand, by the definition of the characteristic function and by the Jensen formula, we have
\begin{align*}
N_H(r)&=\int_{S(r)}\log |f_sg_t-f_tg_s|\sigma_m\\ 
& \le \int_{S(r)}\log ||f||\sigma_m +\int_{S(r)}\log ||g||\sigma_m \\
&=T_f(r)+T_g(r).
\end{align*}
Combining this and (\ref{5.2}), we obtain
$$ T_f(r)+T_g(r) \ge \sum_{i=1}^qN^{[1]}_{Q_i(f)}(r).$$
Similarly, we have
$$ T_f(r)+T_g(r) \ge \sum_{i=1}^qN^{[1]}_{Q_i(g)}(r).$$
Summing-up both sides of the above two inequalities, we have
\begin{align}\label{5.3}
 2(T_f(r)+T_g(r))& \ge \sum_{i=1}^qN^{[1]}_{Q_i(f)}(r)+\sum_{i=1}^qN^{[1]}_{Q_i(g)}(r).
\end{align}
From (\ref{5.3}) and applying Theorem \ref{1.1} for $f$ and $g$, we have
\begin{align*}
& 2(T_f(r)+T_g(r))\\
&\ge \sum_{i=1}^q\dfrac{1}{H_V(d)-1}N^{[H_V(d)-1]}_{Q_i(f)}(r)+\sum_{i=1}^q\dfrac{1}{H_V(d)-1}N^{[H_V(d)-1]}_{Q_i(g)}(r)\\
&\ge \dfrac{d}{H_V(d)-1}\left (q-\dfrac{(2N-k+1)H_V(d)}{k+1}\right )(T_f(r)+T_g(r))+o(T_f(r)+T_g(r)).
\end{align*}
Letting $r\longrightarrow +\infty$, we get 
$$2 \ge \frac{d}{H_V(d)-1}\left (q-\frac{(2N-k+1)H_V(d)}{k+1}\right ),$$
$$\text{i.e., } q\le \frac{2(H_V(d)-1)}{d}+\frac{(2N-k+1)H_V(d)}{k+1}.$$
This is a contradiction.
Hence $f=g$. The assertion a) is proved.

b) Again, by Lemma \ref{5.1}, we have $ ||\ T_f(r)=O(T_g(r))\ \text{ and }\  ||\ T_g(r)=O(T_f(r)).$ Suppose that the assertion b) of the theorem does not hold.

By changing indices if necessary, we may assume that
$$\underbrace{\dfrac{Q_{1}(f)}{Q_{1}(g)}\equiv \cdots\equiv \dfrac{Q_{k_1}(f)}{Q_{k_1}(g)}}_{\text { group } 1}\not\equiv
\underbrace{\dfrac{Q_{k_1+1}(f)}{Q_{k_1+1}(g)}\equiv \cdots\equiv\dfrac{Q_{k_2}(f)}{Q_{k_2}(g)}}_{\text { group } 2}$$
$$\not\equiv \underbrace{\dfrac{Q_{k_2+1}(f)}{Q_{k_2+1}(g)}\equiv \cdots\equiv\dfrac{Q_{k_3}(f)}{Q_{k_3}(g)}}_{\text { group } 3}\not\equiv \cdot\cdot\cdot\not\equiv \underbrace{\dfrac{Q_{k_{s-1}+1}(f)}{Q_{k_{s-1}+1}(g)}\equiv \cdots \equiv\dfrac{Q_{k_s}(f)}{Q_{k_s}(g)}}_{\text { group } s},$$
where $k_s=q.$ 

Since the assertion b) of the theorem does not hold, the number of elements of each group is at most $N$. For each $1\le i \le q,$ we set
\begin{equation*}
\sigma (i)=
\begin{cases}
i+N& \text{ if $i+N\leq q$},\\
i+N-q&\text{ if  $i+N> q$}
\end{cases}
\end{equation*}
and  
$$P_i=Q_i(f)Q_{\sigma (i)}(g)-Q_i(g)Q_{\sigma (i)}(f).$$
Then $\dfrac{Q_i(f)}{Q_i(g)}$ and $\dfrac{Q_{\sigma (i)}(f)}{Q_{\sigma (i)}(g)}$ belong to two distinct groups, and hence $P_i\not\equiv 0$ for every $1\le i\le q.$ It is easy to see that
\begin{align*}
 \nu_{P_i}(z)\ge&\min\{\nu_{Q_i(f)}(z),\nu_{Q_i(g)}(z)\}+\min\{\nu_{Q_{\sigma (i)}(f)}(z),\nu_{Q_{\sigma (i)}(g)}(z)\}\\
&+\sum_{\overset{j=1}{j\ne i,\sigma (i)}}^{q}\min\{\nu_{Q_j(f)}(z),1\}\\
\ge&\sum_{j=i,\sigma (i)}\biggl (\min\{\nu_{Q_j(f)}(z),H_V(d)-1\}+\min\{\nu_{Q_j(g)}(z),H_V(d)-1\}\\
& -(H_V(d)-1)\min\{\nu_{Q_j(f)}(z),1\}\biggl )+\sum_{\overset{j=1}{j\ne i,\sigma (i)}}^{q}\min\{\nu_{Q_j(f)}(z),1\}.
\end{align*}
for all $z$ in $\mathbb C^m$.

Integrating both sides of this inequality, we get
\begin{align}\label{5.4}
\begin{split}
||\ N_{P_i}(r) &\ge \sum_{j=i,\sigma (i)}\biggl (N^{[H_V(d)-1]}_{Q_j(f)}(r)+N^{[H_V(d)-1]}_{Q_j(g)}(r)-(H_V(d)-1)N^{[1]}_{Q_j(f)}(r)\biggl )\\
&+\sum_{\overset{j=1}{j\ne i,\sigma (i)}}^{q}N^{[1]}_{Q_j(f)}(r).
\end{split}
\end{align}
Repeating the same argument as in the proof of Theorem \ref{1.2}, by Jensen's formula and by the definition of the characteristic function, we have
\begin{align}\label{5.5} 
||\ N_{P_i}(r)&\le d(T_f(r)+T_g(r))
\end{align}
From (\ref{5.4}) and (\ref{5.5}), we get
\begin{align*}
||\ d(T_f(r)+T_g(r))&\ge \sum_{j=i,\sigma (i)}\biggl (N^{[H_V(d)-1]}_{Q_j(f)}(r)+N^{[H_V(d)-1]}_{Q_j(g)}(r)-(H_V(d)-1)N^{[1]}_{Q_j(f)}(r)\biggl )\\
&+\sum_{\overset{j=1}{j\ne i,\sigma (i)}}^{q}N^{[1]}_{Q_j(f)}(r).
\end{align*}
Summing-up both sides of this inequality over all $1\le i\le q$, we obtain
\begin{align*}
||\ dq(T_f(r)+T_g(r))&\ge 2\sum_{j=1}^{q}\biggl (N^{[H_V(d)-1]}_{Q_j(f)}(r)+N^{[H_V(d)-1]}_{Q_j(g)}(r)\biggl )+(q-2H_V(d))\sum_{j=}^{q}N^{[1]}_{Q_i(f)}\\
&\ge 2d\left (q-\dfrac{(2N-k+1)H_{V}(d)}{k+1}\right )\biggl (T_f(r)+T_g(r)\biggl )+o(T_f(r)).
\end{align*}
Letting $r\longrightarrow +\infty$, we get 
\begin{align*}
&dq \ge  2d\left (q-\dfrac{(2N-k+1)H_{V}(d)}{k+1}\right ),\\
\text{i.e., }&q\le \dfrac{2(2N-k+1)H_{V}(d)}{k+1}.
\end{align*}
This is a contradiction. 

Hence the assertion b) holds. The theorem is proved.\hfill$\square$

\end{document}